\documentclass[12pt]{amsart}
\usepackage{amsmath,amsthm,amsfonts,amssymb,amscd}
\usepackage{amsrefs}
\usepackage[mathscr]{euscript}
\usepackage[utf8]{inputenc}
\usepackage{hyperref}
\usepackage[usenames,dvipsnames]{color}

\newtheorem{theorem}{Theorem}
\newtheorem*{theorem*}{Theorem}
\newtheorem{proposition}[theorem]{Proposition}
\newtheorem{lemma}[theorem]{Lemma}

\theoremstyle{definition}
\newtheorem{definition}[theorem]{Definition}
\newtheorem*{definition*}{Definition}
\newtheorem{example}[theorem]{Example}

\theoremstyle{remark}
\newtheorem{remark}[theorem]{Remark}

\newcommand{\re}{\mathrm{Re}}

\newcommand{\binomial}[2]{\ensuremath{\left( \begin{matrix}#1 \\ #2 \end{matrix} \right)}}

\newcommand{\dd}{\ensuremath{\mathrm d}}

\newcommand{\funcao}[5]{\ensuremath{
\begin{array}{rccl}
#1 : & #2 & \rightarrow & #3,\\ & #4 & \mapsto & #5
\end{array}}}

\renewcommand{\dim}{\ensuremath{\mathrm{dim}_{\mathbb C}}}

\author{Gregorio Malajovich}
\thanks{Partially supported by CNPq, CAPES (Brasil)
and by MathAmSud international cooperation grant {\em Complexity}.}
\address{
Departamento de Matemática Aplicada, 
Instituto de Matemática da UFRJ. C.P.68530,
Rio de Janeiro, RJ 
21941-909, Brasil}
\urladdr{http://www.labma.ufrj.br/~gregorio}
\email{gregorio.malajovich@gmail.com}
\title{On the expected number of zeros of nonlinear equations}
\date{June 27,2013}
\begin{document}
\maketitle
\begin{abstract}
This paper investigates the expected number of complex roots of nonlinear equations.
Those equations are assumed to be analytic, and to belong to
certain inner product spaces. Those spaces are then endowed with
the Gaussian probability distribution.
 
The root count on a given domain is proved to be `additive' with
respect to a product operation of functional spaces. This allows
to deduce a general theorem 
relating the expected number of roots for unmixed and mixed 
systems. Examples of root counts for equations 
that are not polynomials nor exponential sums are given at the end.
\end{abstract}
\section{Introduction}

We consider systems of analytic equations
of the form
\[
f_1(\mathbf x) = \cdots = f_n(\mathbf x) = 0
\]
where $\mathbf x$ is assumed to belong to a complex $n$-dimensional 
manifold $M$. Each $f_i$ 
belongs to a {\em fewspace} or 
{\em space of complex fewnomials} $\mathscr F_i$. Fewspaces are
complex inner product spaces of analytic functions with certain properties.
The definition is postponed to Sec.\ref{sec:fewspaces}.

Let $n_M(\mathbf f)$ denote the number of isolated roots of the system
above. More generally, let $n_{\mathscr K}(\mathbf f)$ be
the number of isolated
roots in a subset $\mathscr K \subseteq M$. A consequence of Brouwer's degree theorem
is that when $\mathscr K$ is open, the number $n_{\mathscr K}(\mathbf f)$ is
lower semi-continuous as a function of $\mathbf f$ (details in
~\cite{NONLINEAR-EQUATIONS}*{Ch.3}).
\medskip
\par
When the $\mathscr F_i$ are spaces of polynomials (resp. Laurent polynomials)
and $M=\mathbb C^n$
(resp. $M=(\mathbb C \setminus\{0\})^n$), the number $n_M(\mathbf f)$ is
known to attain its maximum {\em generically}, that is
for all $\mathbf f$ except in a complex codimension 1 (hence measure zero)
variety. 
Bounds for this maximum are known, and some of them are exact.

For instance, let $\mathscr F_A$ be the set of Laurent polynomials
with support $A$, viz.
\[
f(\mathbf x) = \sum_{\mathbf a \in A} f_{\mathbf a} 
x_1^{a_1} x_2^{a_2} \cdots x_n^{a_n},
\]
where $A \subset \mathbb Z^n$ is assumed to be finite and $f_{\mathbf a}
\in \mathbb C$. The inner product in $\mathscr F_A$ is arbitrary.
Let $\mathscr A$ denote the convex hull of $A$.

\begin{theorem}[Kushnirenko~\cite{Kushnirenko}]\label{th:kushnirenko}
Let $f_1, \cdots, f_n \in \mathscr F_A$.
For a generic choice of coefficients $f_{i \mathbf a} \in \mathbb C$,
\[
n_{ (\mathbb C \setminus\{0\})^n } (\mathbf f) = 
n!\ \mathrm{Vol}(\mathscr A)
.
\]
\end{theorem}

The case $n=1$ was known to Newton, and 
$n=2$ was published by Minding~\cite{Minding} in 1841.
A system as above, where all the equations have the same support $A$
is said to be {\em unmixed}. Otherwise, the system is said to be 
{\em mixed}.
The following root count for mixed polynomial systems
 was published by Bernstein~\cite{Bernstein} and is
known as the BKK bound (for Bernstein, Kushnirenko and Khovanskii)~\cite{BKK}:

\begin{theorem}[Bernstein]\label{th:bernstein}
Let $A_1, \dots, A_n \subset \mathbb Z^n$ be finite sets. 
Let $\mathscr A_i$ be the
convex hull of $A_i$. 
For a generic choice of coefficients $f_{i \mathbf a} \in \mathbb C$,
\[
n_{ (\mathbb C \setminus\{0\})^n } (\mathbf f) 
\]
is $n!$ times the coefficient $V$ of $\lambda_1 \dots \lambda_n$ in the
polynomial
\[
\frac{1}{n!}\mathrm{Vol} (\lambda_1 \mathscr A_1 + \cdots + \lambda_n \mathscr A_n).
\]
\end{theorem}
This number $V$ is known as the {\em mixed volume} of the tuple of convex
bodies $(\mathscr A_1, \dots, \mathscr A_n)$. 
\medskip
\par
Recently, Kaveh and Khovanskii ~\cite{KKh2010, KKh2012} 
generalized those results to the situation where
$M$ is a variety or a quasi-projective variety and each space $\mathscr F_i$ is
a certain subspace of regular functions of $M$. 
\medskip
\par
The objective of this paper is to extend the results above to more
general spaces of analytic equations. For instance, we would like 
to count zeros of equations such as
\begin{equation}\label{example1}
\begin{split}
f_{00} + f_{01}x + f_{02}x^2 + \cdots + f_{0d} x^d
+\hspace{7em}&\\
+ f_{10}e^x + f_{11}xe^x + f_{12}x^2e^x + \cdots + f_{1d} x^de^x
& = 0.
\end{split}
\end{equation}

It is easy to see that the number of solutions in $\mathbb C$ for
(say) $d=0$ is infinite. However, we can inquire about the number of
solutions in a smaller set, like the disk $\mathscr D= \{ x \in \mathbb C:
|x|<1\}$.  

The {\em generic} number of zeros exists no more. Instead,
we endow the space of equations with a probability measure defined as
follows: we assume the space of equations $\mathscr F$ to be an inner product space
of {\em complex} dimension $\dim(\mathscr F)$. 
Then the standard 
Gaussian (or {\em normal}) distribution is given by the probability density function
\[
\frac{1}{(\pi)^{\dim(\mathscr F)}} e^{-\|f\|^2} \mathrm d\mathscr F(f)
.
\]
Alternatively, assuming an orthonormal coordinate system on $\mathscr F$, the 
coordinates
of $f$ are iid complex Gaussians with zero average and unit variance.
Assuming this probability density function for the {\em random variable} $f$, we 
compute the {\em expected} number of isolated roots.

In the example above, the expected root count is 
\[
\mathbb E (n_{\mathscr D} (f)) = d/2 + 
0.202,918,921,282 \cdots
\]
(see Section~\ref{sec:product} for the precise inner product we are using).
The constant $0.202 \cdots$ was obtained numerically. I would like
to thank Steven Finch for pointing out an error in the 4-th
decimal of a previous
computation, and giving the correct decimal expansion. 

This and other examples are worked out in Section~\ref{sec:examples}
\medskip
\par

It turns out that complex fewnomial spaces are reproducing kernel spaces.
A meaningful multiplication operation
between reproducing kernel spaces was studied by Aronszajn~\cite{Aronszajn}
(see Section~\ref{sec:product}). We denote the
product space of $\mathscr F$ and  $\mathscr G$ by $\mathscr F \mathscr G$.
If $\lambda \in \mathbb N$, we denote
the $\lambda$-th power of $\mathscr F$ by $\mathscr F^{\lambda}$.
The main result in this paper is an analog to Bernstein's theorem.
However, there is no more an interpretation of the number of roots in
terms of a volume of a convex body (Minding, Kushnirenko, or Okounkov~\cite{Okounkov,
KKh2012}) or
in terms of mixed volume. But the relation between root counts in
mixed and unmixed systems is preserved. 

\begin{theorem}\label{th:main}
Let $\mathscr F_1, \dots, \mathscr F_n$ be finite dimensional 
fewspaces of functions of $M$, endowed with 
standard Gaussian probability distributions. Let $\mathscr K \subseteq M$ be measurable. Then,
\[
\mathbb E_{
f_1 \in \mathscr F_1,
\dots,
f_n \in \mathscr F_n}
(n_{\mathscr K} (\mathbf f))
\]
is the coefficient of $\lambda_1 \lambda_2 \cdots \lambda_n$ in
the $n$-th degree homogeneous polynomial
\[
\frac{1}{n!}
\mathbb E_{
g_1, \dots, g_n
\in \mathscr F_1^{\lambda_1} \mathscr F_2^{\lambda_2} \cdots \mathscr F_n^{\lambda_n}}
(n_{\mathscr K} (\mathbf g))
\]
where the standard Gaussian probability distribution
is assumed in each
$\mathscr F_1^{\lambda_1} \mathscr F_2^{\lambda_2} \cdots \mathscr F_n^{\lambda_n}$.
\end{theorem}

In the setting of Bernstein's theorem, one may identify 
$\mathscr F_{\lambda_1 A_1 + \cdots + \lambda_n A_n}$ 
to $\mathscr F_{A_1}^{\lambda_1} \mathscr F_{A_2}^{\lambda_2} \cdots \mathscr F_{A_n}
^{\lambda_n}$. 
(See Section~\ref{sec:product} for details).
With this identification, Bernstein's theorem follows
immediately from Kushnirenko's theorem and Theorem~\ref{th:main}.
\medskip
\par

A basic step for the proof of Theorem~\ref{th:main} is:

\begin{lemma}\label{lemma:main}
Let $\mathscr E, \mathscr F_1, \mathscr F_2, \dots, \mathscr F_n$ be finite dimensional fewspaces
of functions of $\mathscr M$. Let $\mathscr K \subseteq M$ be measurable. 
Then,
\begin{eqnarray*}
\mathbb E_{
f_1 \in 
\mathscr E
\mathscr F_1,
f_2 \in \mathscr F_2,
\dots,
f_n \in \mathscr F_n}
(n_{\mathscr K} (\mathbf f))
&=&
\mathbb E_{
f_1 \in \mathscr E,
f_2 \in \mathscr F_2,
\dots,
f_n \in \mathscr F_n}
(n_{\mathscr K} (\mathbf f))
+\\
& &+
\mathbb E_{
f_1 \in \mathscr F_1,
f_2 \in \mathscr F_2,
\dots,
f_n \in \mathscr F_n}
(n_{\mathscr K} (\mathbf f))
.
\end{eqnarray*}
Above, all fewspaces are assumed with the standard Gaussian
probability distribution.
\end{lemma}

\begin{remark}
This generalizes a result by Kaveh and Khovanskii~\cite{KKh2010}*{Th.4.23}, 
where $\mathscr K$ is
assumed to be an irreducible $n$-dimensional complex variety. There is however a
subtle difference. As they
consider certain semi-group of linear spaces 
of {\em regular functions} on $\mathscr K$, the choice of inner products is
immaterial. Here, this choice matters. This is why
we insist in speaking about a product of {\em inner product} linear spaces, where
the  inner product at $\mathscr E \mathscr F_1$ is determined by the inner products
at factors $\mathscr E$ and $\mathscr F_1$. This inner product appeared in
Aronszajn's paper \cite{Aronszajn}, in connection with reproducing kernel spaces
(Section~\ref{sec:product}).
\end{remark}

\section{Related work}

Random polynomial systems constitute a classical subject of studies, 
and received
a lot of attention lately (See for instance the book by Azaïs and
Wschebor~\cite{Azais-Wschebor} and references). Part of the interest
comes from the study of algorithms for solving polynomial systems
such as in~\cites{Bezout1,Bezout2,Bezout3,Bezout4,Bezout5}.
The running time of algorithms can be estimated in terms of certain
invariants, such as the number of real or complex zeros,
and the condition number. While the number of real zeros of real
polynomial systems and the condition number depend on the input system,
it is possible to obtain probabilistic complexity estimates by
endowing the space of polynomials with a probability distribution, and
then treating those quantities as random variables. For the full picture,
see the book~\cite{BCSS} and two forthcoming books~\cites{NONLINEAR-EQUATIONS,
CONDITIONNING}.
Recent papers on the subject include
~\cites{ Dedieu-Malajovich2008, Armentano-Dedieu2009, CKMW2, CKMW3}.
The extension of this theory to systems of sparse polynomial systems
started with~\cites{Malajovich-Rojas2002,Malajovich-Rojas2004} (see below)
and is still a research subject (see~\cite{NONLINEAR-EQUATIONS}).
\medskip
\par
The subject of random (Gaussian) analytic functions in one complex
variable is also quite active (see for instance the book \cite{HKPV} or
the review \cite{Nazarov-Sodin}). Of particular interest are families of
random functions with density of zeroes invariant by a convenient group of
transformations~\cite{Sodin, Sodin-Tsirelson}.
There are connections with determinantal (or {\em fermionic}) point processes
~\cite{Peres-Virag, HKPV} and zeros ({\em nodal points}) of random spherical harmonics~\cite{Nazarov-Sodin}.

\medskip
\par
Another source of interest comes from classical
asymptotic estimates such as
in  
Littlewood-Offord~\cites{Littlewood-Offord1943,Littlewood-Offord1945}
and Kac~\cites{Kac1943,Kac1949}. 
\par

Asymptotic formulas for the number of roots of sparse polynomial systems
can be obtained by scaling the supports. For instance, one looks
at systems of Laurent polynomials such as
\[
f_i(\mathbf x) = \sum_{\mathbf a \in A_i} f_{i\mathbf a} \mathbf x^{t \mathbf a}
\]
where $t$ is a scaling parameter. A random variable of
interest in the zero-dimensional case 
is $t^{-n} n_{M}(\mathbf f)$.  
In~\cite{Shiffman-Zelditch}, Shiffman and Zelditch gave asymptotic
formulas for the root density in terms of the mixed volume form.

Kazarnovskii~\cite{Kazarnovskii2004} obtained more general formulas.
He considered 
fewnomials that are (after multiplying variables by $\sqrt{-1}$) 
Fourier transforms of distributions supported by real compact sets.
For instance, \eqref{example1} is the Fourier transform of a distribution
with support $\{0,1\}$, namely
\[
\sum_{\substack{i=0,1\\ j=0,\dots,d}} f_{ij} \frac{(-1)^j}{j!} \delta_i^{(j)}(y) 
.
\]

The convex bodies that appear in the Kushnirenko and Bernstein theorems
are replaced by the convex hull of the support of the distributions.
In this sense, he generalized Bernstein's theorem to non-polynomials
and non-exponential-sums. However, his bounds for (say) \eqref{example1}
do not take into account different values of $d$. That is why those bounds
must be asymptotic.

More recently, Kaveh and Khovanskii\cite{KKh2010} developed an intersection
theory for spaces of rational functions over irreducible projective varieties.
The {\em complex fewnomial spaces} introduced in this paper are an attempt to generalize
some of this theory to more general spaces of holomorphic functions.

A generalization of the Newton polytope introduced by Okounkov~\cite{Okounkov}
(the {\em Newton-Okounkov body}) plays an important role in the intersection
theory for rational functions~\cite{KKh2012}.

\section{Spaces of complex fewnomials}
\label{sec:fewspaces}
Let $M$ be an $n$-dimensional complex manifold. In this section we
review part of the theory of spaces of complex fewnomials
in $M$. 
Some further details can be found  
in~\cite{NONLINEAR-EQUATIONS}. Canonical references for
analytic functions of several variables and for reproducing
kernel spaces are, respectively, \cite{Krantz} and
~\cite{Aronszajn}.

\begin{definition} \label{fewspace}
A {\em complex fewnomial space} $\mathscr F$ (or {\em fewspace} for short)
of functions over a complex manifold $M$ 
is a Hilbert space 
of holomorphic functions from $M$ to $\mathbb C$
such that the following holds.
Let $V: M\rightarrow \mathscr F^*$
denote the {\em evaluation form} $V(\mathbf x): f \mapsto f(\mathbf x)$.
For any $\mathbf x \in M$, 
\begin{enumerate}
\item \label{rep-space} $V(\mathbf x)$ is a continuous linear form.
\item \label{nevervanish} $V(\mathbf x)$ is not the zero form.
\end{enumerate}
In addition, we say that the fewspace is {\em non-degenerate} if 
and only if, for any $\mathbf x \in M$, 
\begin{itemize}
\item[(3)] $P_{V(\mathbf x)} DV(\mathbf x)$ has full rank, 
\end{itemize}
where
$P_W$ denotes the orthogonal projection onto $W^{\perp}$. (The derivative
is with respect to $\mathbf x$).
In particular, a non-degenerate fewspace has complex dimension $\ge n+1$.
\end{definition}

\begin{remark} Eventual points $x \in M$ such that $V(x)=0$ are known as
{\em base locus} of $\mathscr F$. In that language, condition (2) says that
$\mathscr F$ has no base locus.
\end{remark}

\begin{example} Let $M$ be an open connected subset of $\mathbb C^n$.
{\em Bergman space} $\mathscr A(M)$ is the space of holomorphic functions
defined in $M$ with finite $\mathscr L^2$ norm. The inner product is
the $\mathscr L^2$ inner product. When $M$ is bounded, $\mathscr A(M)$
contains constant and linear functions, hence it is a 
non-degenerate fewspace. For more details, see~\cite{Krantz}.
\end{example}

\begin{remark}
Condition 1 holds trivially for any finite dimensional fewnomial space,
and less trivially for subspaces of Bergman space. 
\end{remark}

\begin{example} If $M$ is a quasi-projective variety, then $\mathscr O(M)$ denotes
the space of regular (holomorphic) functions $M \rightarrow \mathbb C$. 
Let $\mathscr L$ be a finite dimensional subspace of $\mathscr O(M)$ 
without base locus. Then $\mathscr L$ is a fewspace. The semi-group
$K_{\mathrm {reg}}$ of all such spaces plays an important role in the
intersection theory of Kaveh and Khovanskii~\cite{KKh2010}. 
\end{example}

To each fewspace $\mathscr F$  we associate two objects: The  
{\em reproducing kernel} $K(\mathbf x, \mathbf y)
=K_{\mathscr F}(\mathbf x,\mathbf y)$ (also known as the 
{\em covariance kernel}) and a possibly degenerate Kähler form
$\omega = \omega_{\mathscr F}$ on $M$.

Item~(\ref{rep-space}) in the definition makes $V(\mathbf x)$ an
element of the dual space $\mathscr F^*$ of $\mathscr F$ (more precisely,
the space of continuous functionals $\mathscr F \rightarrow \mathbb C$).

Riesz-Fréchet representation Theorem 
(e.g. \cite{Brezis} Th.V.5 p.81)
allows to identify $\mathscr F$ and $\mathscr F^*$.
Let $V(\mathbf x)^* \in \mathscr F$ be the dual element 
of $V(\mathbf x) \in \mathscr F^*$. Then we define the 
Kernel 
$K(\mathbf x,\mathbf y) = \overline{(V(\mathbf x)^*)(\mathbf y)}$. 
For fixed $\mathbf x$, $\mathbf z \mapsto K(\mathbf x, \bar{\mathbf z}) \in
\mathscr F$.
\medskip
\par

Let $\langle \cdot , \cdot \rangle$ denote the inner product in $\mathscr F$. By convention,
it is linear in the first variable and antilinear in the second variable.
By construction, for $f \in \mathscr F$,
\[
f(\mathbf y) = 
\langle f(\cdot), K(\cdot, \mathbf y) \rangle
.
\]

There are two consequences. First of all,
\[
K(\mathbf y,\mathbf x) = 
\langle K(\cdot, \mathbf x), K(\cdot, \mathbf y) \rangle
=
\overline{\langle K(\cdot, \mathbf y), K(\cdot, \mathbf x) \rangle}
= \overline{K(\mathbf x,\mathbf y)}
\]
and in particular, for any fixed $\mathbf y$, $\mathbf x \mapsto K(\mathbf x,\mathbf y)$ is also an
element of $\mathscr F$. Thus, $K(\mathbf x,\mathbf y)$ is analytic in $\mathbf x$ and in $\bar {\mathbf y}$.
Moreover, $\| K(\mathbf x, \cdot) \|^2 = K(\mathbf x,\mathbf x)$.

Secondly, $Df(\mathbf y) \dot {\mathbf y} = 
\langle f(\cdot), D_{\bar {\mathbf y}} K(\cdot, {\mathbf y}) \bar{\dot {\mathbf y}} \rangle$ and the same holds for higher derivatives.

Because of Definition~\ref{fewspace}(\ref{nevervanish}),
$K(\cdot, y) \ne 0$. 
Thus, $y \mapsto K(\cdot, y)$ induces a map from $M$ to $\mathbb P(\mathscr F)$.

\begin{remark} In \cite{KKh2010}, 
the corresponding map $M \rightarrow \mathbb P(\mathscr F^*)$ is also known
as the {\em Kodaira map}.
\end{remark}

Let $\partial$ and $\bar \partial$ denote the holomorphic (resp. anti-holomorphic)
exterior derivative operators \cite{Krantz}. 
Recall that the Fubini-Study form 
\[
\omega_f 
=
\frac{\sqrt{-1}}{2} \partial \bar \partial \log \|f\|^2
\]
is defined in $\mathscr F \setminus \{0\}$ and induces
a non-degenerate symplectic 1-1 form on $\mathbb P(\mathscr F)$.
The differential form $\omega_{\mathscr F}$ is defined on $M$ as the pull-back
of the Fubini-Study form by $y \mapsto K(\cdot, y)$. As $\partial$ and $\bar \partial$
commute with the pull-back operator,
\begin{equation}\label{fubini}
(\omega_{\mathscr F})_{\mathbf x} = 
\omega_{\mathbf x} = \frac{\sqrt{-1}}{2} \partial \bar \partial  \log K(\mathbf x,\mathbf x) .
\end{equation}

When the form $\omega$ is non-degenerate for all $x \in M$,
it induces a Hermitian structure on $M$. This happens if and
only if the fewspace is a non-degenerate fewspace.

\begin{remark} If $\phi_i(x)$ denotes an 
orthonormal basis of $\mathscr F$ (finite or infinite), then
the kernel can be written as
\[
K(\mathbf x,\mathbf y) = \sum \phi_i(\mathbf x) \overline{\phi_i(\mathbf y)} .
\] 
\end{remark}

\begin{remark} If $\mathscr F=\mathscr A(M)$ is the Bergman space, the kernel
obtained above is known as the {\em Bergman Kernel} and the
metric induced by $\omega$ as the {\em Bergman metric}. 
\end{remark}

\begin{remark}
It is possible to consider Gaussian analytic functions (GAF) as Gaussian functions in
$\mathscr A(\mathscr D)$. However, it is necessary to pay some attention to
the covariance.
Most recent results refer to Gaussian functions with
a particular (diagonal) covariance bilinear form, with nice invariance properties.
For instance, the GAF in \cite{Peres-Virag} is
\begin{equation}\label{GAF}
f(z)=\sum_{n \ge 0} a_n z^n,
\hspace{3em} a_n \in N(0,1; \mathbb C).
\end{equation}
In orthonormal coordinates $\phi_n(z) = \sqrt{\frac{n+1}{\pi}} z^n$, 
$f(z) = \sum_{n \ge 0} \tilde a_n \phi_n(z)$ with variance 
$E(|\tilde a_n|^2) = \frac{\pi}{n+1}$.
The Gaussian Entire Function~\cite{Nazarov-Sodin} is
\begin{equation}\label{GEF}
g(z)=
\sum_{n \ge 0} b_n \frac{z^n}{\sqrt{n!}} = \sum_{n \ge 0} \tilde b_n \phi_n(z)
,
\hspace{3em} b_n \in N(0,1; \mathbb C)
\end{equation}
so the variance of $\tilde b_n = \frac{b_n\sqrt{\pi}}{\sqrt{n!(n+1)}}$
is $E(|\tilde b_n|^2) = \frac{\pi}{n!(n+1)}$.
\end{remark}

Let $n_{\mathscr K}(f)$ be the
number of isolated zeros of $f$ that belong to a measurable
set $\mathscr K$. The following result
is well-known. It appears in~\cite{Kazarnovskii1984}*{Prop.3} 
and~\cite{GROMOV}*{Prop-Def.1.6A}.
It is a consequence of {\em Crofton's formula}, and can also be
deduced from the 
{\em Rice formula} \cite{Azais-Wschebor} or from the {\em coarea formula} \cite{BCSS}.

\begin{theorem}[Root density]\label{th:density}
Let $\mathscr K$ be a measurable set of an $n$-dimensional complex 
manifold $M$. 
Let $\mathscr F_1, \dots, \mathscr F_n$ be 
fewspaces over $M$.
Let $\omega_1, \dots, \omega_n$ be the induced Kähler forms on $M$.
Assume that $\mathbf f=f_1, \dots, f_n$ 
is a standard Gaussian random variable
 in $\mathscr F = \mathscr F_1 \times \cdots \times \mathscr F_n$.
Then, 
\[
\mathbb E(n_{\mathscr K}(\mathbf f))  
=
\frac{1}{\pi^n}
\int_{\mathscr K} 
\omega_1 \wedge \cdots \wedge \omega_n
.
\]
\end{theorem}

As the formulation in terms of reproducing kernel spaces is not
standard, we sketch the proof below (more details are available
in~\cite{NONLINEAR-EQUATIONS}*{Th.5.11}).

\begin{proof}
First of all, let $\mathscr V = \{ (\mathbf f, \mathbf x) \in
\mathscr F \times \mathscr K: \mathbf f(\mathbf x)=0 \}$ be
the incidence locus, and $\pi_1: \mathscr V \rightarrow \mathscr F$,
$\pi_2: \mathscr V \rightarrow \mathscr K$ be the canonical 
projections.

In a neighborhood of each regular point $(\mathbf f_0, \mathbf x_0)$ 
of $\pi_1$, it is possible to parametrize $\mathscr V$ by an
implicit function $(\mathbf f, G(\mathbf f))$ with $G(\mathbf f_0)=\mathbf x_0$
and
\[
DG (\mathbf x_0) = - D\mathbf f(\mathbf x_0)^{-1} 
\left( K_1(\cdot, \mathbf x_0)^* 
\oplus
\cdots
\oplus
K_n(\cdot , \mathbf x_0)^* 
\right)
\]
where $K_i$ is the reproducing kernel of $\mathscr F_i$.
\par

Recall that the {\em normal Jacobian} $NJ_G(\mathbf f)$ of a submersion 
$G$ at $\mathbf f$ is
the product of the singular values of $DG(\mathbf f)$. 
It is more convenient 
to write $NJ_G(\mathbf f) = \det( DG(\mathbf f) DG(\mathbf f)^* )^{1/2}$. Also,
$NJ(\mathbf f_0, \mathbf x_0)$ denotes the normal Jacobian $NJ_G(\mathbf f_0)$ where
$G$ is the implicit function defined above.

Let $\mathscr F_{\mathbf x}$ denote the product
$K_1(\cdot, \mathbf x)^{\perp} \times \cdots \times
K_n(\cdot, \mathbf x)^{\perp} \subseteq \mathscr F$.
The coarea formula \cite[Th.5 p.243]{BCSS} is now
\begin{eqnarray*}
\mathbb E(n_{\mathscr K}(\mathbf f))  
&=&
\frac{1}{\pi^{\dim (\mathscr F)}} \int_{\mathscr F} 
\# \{(\pi_2 \circ \pi_1^{-1}) (\mathbf f) \}
e^{-\|\mathbf f\|^2} \ \dd V_{\mathscr F}(\mathbf f)
\\
&=&
\frac{1}{\pi^{\dim (\mathscr F)} } 
\int_{\mathscr K} dV_M(\mathbf x)
\int_{\mathscr F_{\mathbf x}}
NJ(\mathbf f,\mathbf x)^{-2}
e^{-\|\mathbf f\|^2} \ \dd V_{\mathscr F_{\mathbf x}}(\mathbf f)
\end{eqnarray*}
with $NJ(\mathbf f, \mathbf x)  
= |\det D\mathbf f(\mathbf x)|^{-1}\prod(K_i(\mathbf x, \mathbf x))^{1/2}$.
\medskip
\par
Let $P$ denote the $2n \times 2n$ {\em shuffling} matrix, that is
\[
P_{ij} = 
\left\{
\begin{array}{ll}
1 & \text{if $i-1 \equiv 2 (j-1) \mod 2n$} \\
0 & \text{otherwise.}
\end{array}
\right.
\]

The Leibniz formula for the determinant yields:
\begin{eqnarray*}
\left| \det D\mathbf f(\mathbf x)\right|^{2} \dd V
&=& \det \left(P \left[
\begin{matrix}
D\mathbf f(\mathbf x) & 0 \\
0 & \overline{D\mathbf f(\mathbf x)} 
\end{matrix}\right] P^T\right)
 \dd V\\
&=&
\bigwedge_{i=1}^{n}
\sum_{j,k=1}^n
\frac{\partial}{\partial x_j} f_i(\mathbf x)
\overline{\frac{\partial}{\partial x_k} f_i(\mathbf x)} 
\
\frac{\sqrt{-1}}{2} 
\dd x_j \wedge d\bar x_k .
\end{eqnarray*}
At this point,
\[
\mathbb E(n_{\mathscr K}(\mathbf f))  
=
\frac{1}{\pi^n} 
\int_{\mathscr K} dV_M(\mathbf x)
\bigwedge_{i=1}^n \Omega_i
\]
with 
\[
\Omega_i = 
\int_{K_i(\cdot, \mathbf x)^\perp}
\frac{ 
\frac{\partial}{\partial x_j} f_i(\mathbf x)
\overline{\frac{\partial}{\partial x_k} f_i(\mathbf x)} 
\
\frac{\sqrt{-1}}{2} 
\dd x_j \wedge d\bar x_k
}
{K_i(\mathbf x, \mathbf x)}
\frac{
e^{-\|f_i\|^2} 
}{\pi^{\dim (\mathscr F_i) - 1}}
\ \dd V_{K_i(\cdot, \mathbf x})^{\perp}(f_i)
.
\]

Proposition~\ref{prop:pull-back} below implies that $\Omega_i = 2 \omega_i$,
concluding the proof of the density theorem.
\end{proof}

Let $J: T_{\mathbf x}M \rightarrow T_{\mathbf x}M$ be the complex structure of
$M$. In coordinates, it corresponds to the multiplication by $\sqrt{-1}$.
\begin{proposition}\label{prop:pull-back} 
Let $\langle \mathbf u, \mathbf w \rangle_{i,\mathbf x} = \omega_{i,\mathbf x}(\mathbf u,J\mathbf w)$ be the (possibly degenerate)
Hermitian product
associated to $\omega_i$. Then,
\[
\langle \mathbf u, \mathbf w \rangle_{i,\mathbf x} = 
\int_{K_i(\cdot, \mathbf x)^{\perp}}
\frac{ (Df_i(\mathbf x)\mathbf u) \overline{Df_i(\mathbf x)\mathbf w}}
{K_i(\mathbf x,\mathbf x)}
\frac{
e^{-\|f_i\|^2}
}{\pi^{\dim (\mathscr F_i) - 1}}
\ \dd V_{K_i(\cdot, \mathbf x})^{\perp}(f_i)
.
\]
\end{proposition}

\begin{proof}[Proof of Proposition~\ref{prop:pull-back}]
Let   
\[
P_{\mathbf x} =  I - \frac{ K_i(\cdot, \mathbf x)K_i(\cdot, \mathbf x)^* }{K_i(\mathbf x,\mathbf x)}
\]
be the orthogonal projection onto $K_i(\cdot, \mathbf x)^{\perp}$.
Since the inner product $\langle \cdot , \cdot \rangle_i$ is the
pull-back of Fubini-Study by $\mathbf x \mapsto K_i(\mathbf x, \cdot)$,
we can write the
left-hand-side as:
\begin{eqnarray*}
\langle \mathbf u, \mathbf w \rangle_{i,\mathbf x}  
&=&
\frac{
\left \langle
P_{\mathbf x} DK_i(\cdot, \mathbf x) \mathbf u,
P_{\mathbf x} DK_i(\cdot, \mathbf x) \mathbf w
\right \rangle
}
{K_i(\mathbf x,\mathbf x)}
\end{eqnarray*}

For the right-hand-side, note that
\[
Df_i(\mathbf x) \mathbf u = \langle f_i(\cdot) , D K_i(\cdot, \mathbf x)\mathbf u \rangle
= 
\langle f_i(\cdot) , P_{\mathbf x} DK_i(\cdot, \mathbf x) \mathbf u \rangle
.
\]

Let $\mathbf U = \frac{1}{\|K_i(\cdot, \mathbf x)\|}P_{\mathbf x} DK_i(\cdot, \mathbf x) \mathbf u$ and 
$\mathbf W = \frac{1}{\|K(\cdot,\mathbf x)\|}P_{\mathbf x} DK(\cdot, \mathbf x) \mathbf w$.
Both $\mathbf U$ and $\mathbf W$ belong to $\mathscr F_{\mathbf x}$.
The right-hand-side is 
\begin{eqnarray*}
\int_{K_i(\cdot, \mathbf x)^\perp}
\frac{ (Df_i(\mathbf x)\mathbf u) \overline{Df_i(\mathbf x)\mathbf w}}
{\|K_i(\mathbf x,\mathbf x)\|^2}
\
\frac{
e^{-\|f_i\|^2}
}{\pi^{\dim (\mathscr F_i) - 1}}
\ \dd V_{K_i(\cdot, \mathbf x)^{\perp}}(f_i)
\hspace{-20em}
&&
\\
&=&
\int_{K_i(\cdot, \mathbf x)^\perp}
\langle f_i, \mathbf U \rangle
\overline{\langle f_i, \mathbf W \rangle}
\
\frac{
e^{-\|f_i\|^2}
}{\pi^{\dim (\mathscr F_i) - 1}}
\ \dd V_{K_i(\cdot, \mathbf x)^{\perp}}(f_i)
\\
&=&
\langle \mathbf U,\mathbf W \rangle 
\int_{\mathbb C}\frac{1}{\pi} |z|^2 e^{-|z|^2} \ \mathrm d z
\\
&=&
\langle \mathbf U,\mathbf W \rangle 
\end{eqnarray*}
which is equal to the left-hand-side.

\end{proof}

\begin{remark}The proof of Theorem~\ref{th:density} does not require
$\mathscr F$ to be finite dimensional. If $f$ is a standard
Gaussian random variable in $\mathscr F$ for $\mathscr F$ infinite
dimensional, $\mathbb E(f) = \infty$ so $f \not \in \mathscr F$. However
one may still have $f \in \mathscr O(M)$ almost surely, with $K(x,y)$
and $\omega$ well-defined.
\end{remark}
\par
\begin{example}
The (known) root density of
the Gaussian analytic functions in \eqref{GAF} and \eqref{GEF} can
be recovered from Theorem~\ref{th:density}:\\ \noindent
\centerline
{
\begin{tabular}{|c|c|}
\hline
Hyperbolic, $M= \mathscr D$ & Affine, $M = \mathbb C$ \\
\hline
$f(z) = \sum_{n \ge 0} a_n z^n$, &
$g(z) = \sum_{n \ge 0} b_n \frac{z^n}{\sqrt{n!}}$,\\
$a_n \in N(0,1; \mathbb C).$&
$b_n \in N(0,1; \mathbb C).$\\
\hline
\hspace{1em}
$K(x,y) = \sum_{n \ge 0} x^n (\bar y)^n = \frac{1}{1-x \bar y}$ 
\hspace{1em}&
\hspace{1em}
$K(x,y) = \sum_{n \ge 0} \frac{x^n \bar y^n}{n!} = e^{x \bar y}$
\hspace{1em}\\ 
\hline
$\omega_z = \frac{\sqrt{-1}}{2}
\frac{\dd z \wedge \dd \bar z}{(1-z \bar z)^2} $&
$\omega_z = \frac{\sqrt{-1}}{2}
\dd z \wedge \dd \bar z$\\
\hline
Density at $z$: $\frac{1}{\pi (1 - z \bar z)^2}\dd V$.
&
Density at $z$: $\frac{1}{\pi}\dd V$.\\
\hline
\end{tabular}}
\end{example}

\section{Product spaces}
\label{sec:product}

Let $\mathbf E$ and $\mathbf F$ be complex inner product spaces.
If $e \in \mathbf E$ and $f \in \mathbf F$, we denote by
$e \otimes f$ the class of equivalence of pairs $(e,f)$ under
$(\lambda e, f) \sim (e, \lambda f)$. 
The tensor or {\em direct} product of $\mathbf E$ and $\mathbf F$
is the completion of the space of all linear combinations of elements
of the form $e \otimes f$ (See~\cite{Aronszajn} for details).
In the case $\mathbf E$ and $\mathbf F$
are finite dimensional, $\mathbf E \otimes \mathbf F$ is just the
space of 
bilinear maps $\mathbf E^* \times \mathbf F^*
\rightarrow \mathbb C$.

The canonical inner product for the tensor product of two spaces is given by
\[
\left\langle
e_1 \otimes f_1 ,
e_2 \otimes f_2
\right\rangle_{\mathbf E \otimes \mathbf F}
=
\left\langle
e_1,
e_2
\right\rangle_{\mathbf E}
\left\langle
f_1 ,
f_2
\right\rangle_{\mathbf F}
.
\]

Now, let $\mathscr E$ and $\mathscr F$ be fewnomial spaces on
some complex manifold $M$. Then, $\mathscr E \otimes \mathscr F$
is a fewnomial space on the product $M \times M$, where we
interpret $(e \otimes f)(x_1,x_2) = e(x_1) f(x_2)$.
A classical fact on reproducing kernel spaces allows to recover
the kernel of the tensor product:

\begin{theorem}[Aronszajn]\label{th:aronszajn:1}
The direct (=tensor) product $\mathscr E \otimes \mathscr F$
possesses the reproducing kernel
\[
K_{\mathscr E \otimes \mathscr F} \left( (x_1, x_2) , (y_1, y_2) \right)
=
K_{\mathscr E} \left( x_1 , y_1 \right)
K_{\mathscr F} \left( x_2, y_2 \right)
\]
.
\end{theorem}

This is \cite{Aronszajn}*{Theorem I p.361}. 
Theorem II ibid gives us a convenient notion
of `product' for reproducing kernel spaces
with same domain:

\begin{theorem}[Aronszajn]\label{th:aronszajn:2}
The kernel $K_{\mathscr G}(x,y) = 
K_{\mathscr E}(x,y)
K_{\mathscr F}(x,y)$ is the reproducing kernel
of the class $\mathscr G$ of restrictions of all functions
of the direct (=tensor) product $E \otimes F$ to the diagonal set
$M_1 = \{(x,x): x \in M\} \simeq M$. 
For any such restriction, 
$\|g\| = \min \|g'\|_{E \otimes F}$, the restriction of which to the
diagonal set $M_1$ is $g$. 
\end{theorem}

If $\mathscr E$ and $\mathscr F$ are spaces of fewnomials on $M$,
we denote by $\mathscr E \mathscr F$ the class $\mathscr G$
described above. As an inner product space, 
$\mathscr G$ is just the orthogonal complement of the kernel
of the restriction operator
\[
\funcao{\Delta = \Delta_{\mathscr E, \mathscr F}}{\mathscr E \otimes \mathscr F}{\mathscr E \mathscr F \subseteq \mathscr O(M)}{g'}{g=g'_{|M_1} } 
\]
The inner product of $\mathscr G$ is by definition the inner product of
$\mathscr E \otimes \mathscr F$ restricted to $(\ker \Delta)^{\perp}$. 

\begin{lemma}
Let $M$ be fixed. The product of fewspaces of $M$
is associative and commutative. If one introduces
the `constant' fewspace $\mathscr I = \{1\}$, then
the set of fewspaces on $M$ is a commutative semigroup.
\end{lemma}

\begin{proof}
The only nontrivial property to check is associativity.
The space $\mathscr E (\mathscr F \mathscr G)$ is generated by
all $h = e f g$ with $e \in \mathscr E, f \in \mathscr F, g \in \mathscr G$,
so $\mathscr E(\mathscr F \mathscr G) = (\mathscr E \mathscr F) \mathscr G$ as
a linear space. It remains to check that those spaces have the same norm or inner product.
\par
Let $\Delta_{\mathscr E,\mathscr F, \mathscr G}: \mathscr E \otimes \mathscr F \otimes \mathscr G
\rightarrow \mathscr E \mathscr F \mathscr G$ be the restriction to the diagonal
$\Delta_{\mathscr E, \mathscr F, \mathscr G}: h(x,y,z) \rightarrow h(x,x,x)$.
Let $h \in \mathscr E \mathscr F \mathscr G$. Let $h' \in (\Delta_{\mathscr E,\mathscr F, \mathscr G})^{-1}(h)$ be the vector with minimal norm. Assume that
\[
h' = \sum_{l \in \Lambda} e_l \otimes f_{l} \otimes g_{l} .
\]
with the system $(e_l \otimes f_l \otimes g_l)_{l \in \Lambda}$ orthogonal in $\mathscr E \otimes \mathscr F \otimes \mathscr G$.
\par
For each $l \in \Lambda$, $e_l \otimes f_{l} \otimes g_{l}$ has minimal norm in 
$(\Delta_{\mathscr E,\mathscr F, \mathscr G})^{-1}(e_l f_l g_l)$. 
In particular, $e_l \otimes (f_l g_l)$ has minimal norm in
$(\Delta_{\mathscr E, \mathscr F \mathscr G})^{-1} (e_l f_l g_l)$ and
$f_l \otimes g_l$ has minimal norm in 
$(\Delta_{\mathscr F, \mathscr G})^{-1} (f_l g_l)$. Thus,
\[
\| e_l (f_l g_l) \|_{\mathscr E (\mathscr F \mathscr G)} = \| e_l \|_{\mathscr E} \| f_l\|_{\mathscr F} \|g_l \|_{\mathscr G}  
.\]
Similarly,
\[
\| (e_l f_l) g_l) \|_{(\mathscr E \mathscr F) \mathscr G} = \| e_l \|_{\mathscr E} \| f_l\|_{\mathscr F} \|g_l \|_{\mathscr G}  
.\]
Thus, $\|h\|_{\mathscr E (\mathscr F \mathscr G)} = \|h\|_{(\mathscr E \mathscr F)\mathscr G}$.  
\end{proof}

\medskip
\par
Given orthonormal bases $(e_a)_{a \in A}$ and $(f_b)_{b \in B}$ of
$\mathscr E$ and $\mathscr F$, we can produce an orthonormal basis
of $\mathscr E \mathscr F$ by a standard Gram-Schmidt argument. However,
in many interesting cases, there is a more explicit formula.

Suppose that those bases have the property that
\begin{equation}\label{transcendental}
e_a \otimes f_b
\perp
e_{a'} \otimes f_{b'}
\
\Rightarrow
\
\left\{
\begin{minipage}{14em}
$\Delta(e_a \otimes f_b)$ 
and 
$\Delta(e_{a'} \otimes f_{b'})$
either orthogonal or colinear.\end{minipage}\right.
\end{equation}
This holds when $\mathscr E$ and $\mathscr F$ are given orthonormal monomial bases. See Sec.\ref{sec:examples}
for non-monomial examples.
.

Let $M$ be path-connected. 
Assume \eqref{transcendental} holds. Then, an orthonormal
basis for $\mathscr E \mathscr F$ is given as follows. Let 
$(a,b) \sim (a',b')$ whenever 
$\Delta(e_a \otimes f_b)$ 
and 
$\Delta(e_{a'} \otimes f_{b'})$ are colinear. We cannot have
$\Delta(e_a \otimes f_b) = 0$ for otherwise $e_a \equiv 0$,
$f_b \equiv 0$ or $M$ is disconnected. Therefore, $\sim$ is
an equivalence relation. 
For every $(a,b)$, choose a root of unity $\omega_{a,b}$ such that for all $c$,
\[
\| \sum_{(a,b) \in c} \omega_{a,b} e_a f_b \|
=
\sum_{(a,b) \in c} \|\omega_{a,b} e_a f_b \|
\]

For every equivalence class $c$ by $\sim$,
set
\begin{equation} \label{product-basis}
g_c = \frac
{\sum_{(a,b) \in c} \| e_a \otimes f_b \| \omega_a e_a f_b}
{\sum_{(a,b) \in c} \| e_a \otimes f_b \|^2}
\end{equation}
and clearly $(g_c)_{c \in C}$ is orthonormal, where $C=(A \times B)/ \sim$.

\begin{example}
$M=\mathbb C^2$,
$e_1 = x$, $e_2 = y$, $f_1 = ix$, $f_2 = y$,  
$g_{\{(1,1)\}} = ix^2$, 
$g_{\{(1,2), (2,1)\}} = \frac{1}{2}xy + \frac{-i}{2}ix^2$, 
$g_{\{(2,2)\}} = y^2$. 
\end{example}

We proved that:
\begin{lemma}\label{lem:basis:product} Assume that condition \eqref{transcendental} holds. Then, $(g_c)_{c \in C}$ given in \eqref{product-basis}
is an orthonormal basis of $\mathscr G$.
\end{lemma}

Here is an example where \eqref{transcendental} fails.
\begin{example}
Let $M=\mathbb C$. The orthonormal bases for spaces $\mathscr E$ and $\mathscr F$
will be, respectively, $(1, x)$ and $(1, 1+x)$. The kernel of $\Delta$ is spanned
by $1 \otimes 1 + x \otimes 1 - 1 \otimes (1+x)$. 
\[
\begin{array}{l|l|l}
\hline
e \otimes f       & \text{Projection onto $\ker \Delta^{\perp}$} & \Delta(e \otimes f) \\
\hline
1 \otimes 1       & \frac{2}{3}(1 \otimes 1) - \frac{1}{3} (x \otimes 1) + \frac{1}{3} (1 \otimes (1+x))    & 1\\
x \otimes 1       & - \frac{1}{3}(1 \otimes 1) + \frac{2}{3} (x \otimes 1) + \frac{1}{3} (1 \otimes (1+x))  & x\\
1 \otimes (1+x)   & \frac{1}{3}(1 \otimes 1) + \frac{1}{3} (x \otimes 1) + \frac{2}{3} (1 \otimes (1+x))    & 1+x\\
x \otimes (1+x)   & x \otimes (1+x)                                                                   & x+x^2. \\
\end{array}
\]
Above, $1 \otimes 1 \perp 1 \otimes (1+x)$ but $\langle 1, 1+x \rangle = 1/3$.
\end{example}

\begin{example}\label{Pd}
Let $M=\mathbb C^n$ and let $\mathscr P_1$ be
the space of affine functions in $n$ variables.
To make it an inner product space, we assume that
$(1, x_1, x_2, \cdots, x_n)$ is an orthonormal
basis.
We define inductively $\mathscr P_{d+1}= \mathscr P_d \mathscr P_1$.
Using Lemma~\ref{lem:basis:product}, we obtain an orthonormal basis of $\mathscr P_d$:
\[
\left(
\sqrt{\binomial{d}{a_0 a_1, \dots, a_n}} 
x_1^{a_1} x_2^{a_2} \cdots x_n^{a^n} 
\right)_{\substack{ a_0, \dots, a_n \ge 0 \\
\sum_{0 \le j \le n} a_j = d}
}
\]
Above, the multinomial coefficient
\[
\binomial{d}{a_0, a_1, \dots, a_n} = 
\frac{d!}{a_0! \cdots a_n!}\]
is the number of ways
to distribute $d=a_0+\cdots+a_n$ balls into $n+1$ numbered buckets of size $a_0$,
\dots, $a_n$. It is also the coefficient of
$x_1^{a_1} x_2^{a_2} \cdots x_n^{a^n}$ in
$(1+x_1 + \cdots + x_n)^d$.
This corresponds to the unitarily invariant inner product defined
by Weyl~\cite{Weyl}, also known as Bombieri's.

The reproducing kernel of $\mathscr P_d$ is easily seen
to be
\[
K_d( \mathbf x, \mathbf y) = (1 + x_1 \bar y_1 + \cdots + x_n \bar y_n)^d .
\]
\end{example}

With the same formalism, we can also retrieve the multi-unitarily invariant
inner product for the space of roots of multihomogeneous polynomial
systems introduced by Rojas~\cite{Rojas}.

\begin{example}
Let $A \subseteq (\mathbb Z)^n$ be finite, and
$M = (\mathbb C_{\ne 0})^n$.
Let $c_{\mathbf a} > 0$ be arbitrary.
Let $\mathscr F_A$ be the space of Laurent polynomials of
the form 
\[
f(\mathbf x) = \sum_{\mathbf a \in A} f_{\mathbf a} \mathbf x^{\mathbf a} 
\]
with the inner product that makes 
$(c_{\mathbf a}^{-1/2} x^{\mathbf a})_{\mathbf a \in A}$
an orthonormal basis. 
Then,
\[
K_A( \mathbf x, \mathbf y ) = \sum_{\mathbf a \in A} 
c_{\mathbf a} \mathbf x^{\mathbf a} \bar{\mathbf y}^{\mathbf a}
.
\]
The $\lambda$-th power $\mathscr F_A^{\lambda}$ 
of $\mathscr F_A$ is precisely $\mathscr F_{B}$
constructed as follows: for each 
$\mathbf b \in \lambda \ \mathrm{Conv}(A) \cap \mathbb Z^n$, define the
weights
\[
c_{\mathbf b} = \sum_{ \mathbf a_1 + \cdots + \mathbf a_{\lambda} = \mathbf b}
c_{\mathbf a_1} c_{\mathbf a_2} \cdots c_{\mathbf a_\lambda}
.
\]
Then $B = \{ \mathbf b \in  \lambda \ \mathrm{Conv}(A) \cap \mathbb Z^n: c_{\mathbf b} \ne 0\}$.
By repeated application of Lemma~\ref{lem:basis:product}, an orthonormal basis for $\mathscr F_B$ is $( c_{\mathbf b}^{-1/2} x^{\mathbf b})$.
\par
One can interpret the weights $c_{\mathbf b}$ as follows. Let $g \in \mathscr F_B$ be
a standard Gaussian random variable. One can write
\[
g(\mathbf x) = \sum_{\mathbf b \in B} g_{\mathbf b} \mathbf x^{\mathbf b}
\]
and with this notation, the $g_{\mathbf b}$'s are independently distributed random variables
in $N(0, c_{\mathbf b})$.
\end{example}

\begin{example}Now, let $A_1, \dots, A_n$ be finite subsets of $\mathbb Z^n$. Let
$M = (\mathbb C \setminus \{0\})^n$. Let $\mathscr F_i$ be the space of all the
Laurent polynomials with support $A_i$, where we assume inner product
\[
\langle 
\sum_{\mathbf a \in A_i} f_{\mathbf a} \mathbf z^{\mathbf a},
\sum_{\mathbf a \in A_i} g_{\mathbf a} \mathbf z^{\mathbf a}
\rangle_{\mathscr F_i} = 
\sum_{\mathbf a \in A_i} f_{\mathbf a} \overline{g_{\mathbf a}} 
.
\]
Let $\lambda_1, \dots, \lambda_n \in \mathbb N$. Let
$\mathscr G = \mathscr F_1^{\lambda_1} \mathscr F_2^{\lambda_2} 
\cdots \mathscr F_n^{\lambda_n}$. Let $B = \lambda_1 A_1 + \cdots + \lambda_n A_n$.
Then $\mathscr G$ is the space of Laurent polynomials of the form
\[
f(z) = 
\sum_{\mathbf b \in B} f_{\mathbf b} \mathbf z^{\mathbf b}
\] 
with inner product
\[
\langle 
\sum_{\mathbf b \in B} f_{\mathbf b} \mathbf z^{\mathbf b},
\sum_{\mathbf b \in B} g_{\mathbf b} \mathbf z^{\mathbf b}
\rangle_{\mathscr F_i} = 
\sum_{\mathbf b \in B} c_{\mathbf b} f_{\mathbf a} \overline{g_{\mathbf b}} 
\]
and $c_{\mathbf b}$ is the number of (ordered) compositions
\[
\mathbf b = 
a_{11} + \cdots + a_{1\lambda_1} + \cdots +
a_{n1} + \cdots + a_{n\lambda_n}
\]
with $\mathbf a_{ij} \in A_i$. 
\end{example}

\begin{remark} The example above allows to recover Bernstein's Theorem
(Th.\ref{th:bernstein}) from Kushnirenko's theorem~(Th.\ref{th:kushnirenko}).
Let $\mathscr A_i$ denote the convex hull of $A_i$.
By Th.\ref{th:kushnirenko},
the expected number of zeros in $(\mathbb C \setminus \{0\})^n$ of
a standard Gaussian random variable $g \in \mathscr G$ is also the
generic number of zeros, that is
\[
n! 
\mathrm{Vol} (\lambda_1 \mathscr A_1 + \cdots + \lambda_n \mathscr A_n).
\]
Therefore, in Theorem~\ref{th:main}, the expected number of roots is
the coefficient of $\lambda_1 \lambda_2 \cdots \lambda_n$ in the
polynomial
\[
\mathrm{Vol} (\lambda_1 \mathscr A_1 + \cdots + \lambda_n \mathscr A_n).
.
\]
This is exactly $n! V$ where $V$ is the mixed volume of the tuple
$(\mathscr A_1, \dots, \mathscr A_n)$.
\end{remark}

\section{Proof of the main results}

\begin{proof}[Proof of Lemma~\ref{lemma:main}]
Let $\mathscr E$ and $\mathscr F_1$ be fewspaces on a complex manifold
$M$, and let $\mathscr G = \mathscr E \mathscr F_1$.
By Theorem~\ref{th:aronszajn:2},
\[
K_{\mathscr G} (\mathbf x, \mathbf y) = 
K_{\mathscr E} (\mathbf x, \mathbf y)
K_{\mathscr F_1} (\mathbf x, \mathbf y)  
.
\]
By \eqref{fubini}, we deduce that
\[
\omega_{\mathscr G} = \omega_{\mathscr E} + \omega_{\mathscr F_1} .
\] 

Now, we just insert the formula above in Theorem~\ref{th:density}.
\end{proof}

\begin{proof}[Proof of Theorem~\ref{th:main}]
By repeated application of Lemma~\ref{lemma:main},
\[
\begin{split}
\mathbb E_{
g_1, \dots, g_n
\in \mathscr F_1^{\lambda_1} \mathscr F_2^{\lambda_2} \cdots \mathscr F_n^{\lambda_n}}
(n_{\mathscr K} (\mathbf g))
=&\\
&\hspace{-4em}\sum_{i_1=1}^n
\lambda_{i_1}
\mathbb E_{
f_1 \in \mathscr F_{i_1}, 
g_2, \dots, g_n
\in \mathscr F_1^{\lambda_1} \mathscr F_2^{\lambda_2} \cdots \mathscr F_n^{\lambda_n}}
(n_{\mathscr K} (f_1,g_2, \dots, g_n))
.
\end{split}
\]
By the same argument applied to $g_2, \dots, g_n$,
\[
\begin{split}
\mathbb E_{
g_1, \dots, g_n
\in \mathscr F_1^{\lambda_1} \mathscr F_2^{\lambda_2} \cdots \mathscr F_n^{\lambda_n}}
(n_{\mathscr K} (\mathbf g))
=&\\
& \hspace{-2em}\sum_{i_1=1}^n
\cdots
\sum_{i_n=1}^n
\left( \prod_{j=1}^n \lambda_{i_j} \right)
\mathbb E_{
f_1 \in \mathscr F_{i_1}, 
\dots,
f_n \in \mathscr F_{i_n}}
(n_{\mathscr K} (\mathbf f))
.
\end{split}\]
The coefficient in $\lambda_1 \lambda_2 \cdots \lambda_n$ of the expression
above is
\[
n!\ 
\mathbb E_{
f_1 \in \mathscr F_{1}, 
\dots,
f_n \in \mathscr F_{n}}
(n_{\mathscr K} (\mathbf f))
\]
\end{proof}

\section{Explicit calculation of the number of zeros}
\label{sec:examples}

\subsection{The example in the introduction}

We start by the bound on the expected number of roots of
\eqref{example1} in the introduction.
Let $\mathscr E$ denote the fewspace of functions on the disk 
$\mathscr D = \{z \in \mathbb C: |z|<1\}$ spanned by $1$ and $e^z$.
We assume that $1$ and $e^z$ form an orthonormal basis. Then
\[
K_{\mathscr E} (x,y) = 1 + e^{x+\bar y}
.
\]

An easy computation is now
\[
\omega_{\mathscr E} =
\frac{\sqrt{-1}}{2} \partial \bar \partial
\log K_{\mathscr E} (z,z) = \frac{e^{2 \re(z)}}{(1+e^{2 \re(z)})^2}
\]

The following numerical approximation was obtained by Steven Finch
using Mathematica. It was independently checked by this author
using long double IEEE arithmetic. 
\[
\mathbb E_{f \in \mathscr E} (n_f(\mathscr D)) = 
\pi^{-1} \int_{\mathscr D} \ \omega = 
0.202,918,921,282 \cdots
.
\]

The inner product in $\mathscr P_d$ is invariant by the 
reversion operator $f(x) \mapsto x^d f(1/x)$. Hence the standard Gaussian
measure is also invariant, and therefore
$E_{f \in \mathscr P_d}(n_{\mathscr D}(f)) = 
\frac{1}{2} E_{f \in \mathscr P_d}(n_{\mathbb C}(f)) =
 d/2$.
Hence,
\[
\mathbb E_{f \in \mathscr E \mathscr P_d} (n_f(\mathscr D)) = 
\pi^{-1} \int_{\mathscr D} \ \omega  + \omega_{\mathscr P_d} = 
d/2+ 
0.202,918,921,282 \cdots
.
\]

\subsection{An $n$-dimensional example}

We consider now systems where each equation is of the form
\[
\sum f_{\mathbf a, \mathbf b} 
x_1^{a_1} x_2^{a_2} \cdots x_n^{a_n} e^{b_1 x_1 + \cdots + b_n x_n}
\]
and the sum is taken for all $0 \le a_i \le d$ and $b_i=0,1$.
The corresponding domain will be the polydisc $\mathscr D^n$.

The fewnomial space is
\[
\mathscr G = 
(\mathscr E \mathscr P_d) \otimes
(\mathscr E \mathscr P_d) \otimes
\cdots
\otimes
(\mathscr E \mathscr P_d) 
.
\]

Let $\omega = g(z) \ \frac{\sqrt{-1}}{2} \dd z \wedge \dd \bar z$ be
the Kähler form corresponding to $(\mathscr E \mathscr P_d)$.
Then from Th.\ref{th:aronszajn:1} and \eqref{fubini}, we deduce that
\[
\omega_{\mathscr G} = 
\sum_{i=1}^n
g(z_i) \ \frac{\sqrt{-1}}{2} \dd z_i \wedge \dd \bar z_i
.
\]

Hence,
\begin{eqnarray*}
\mathbb E_{f_1, \dots, f_n \in \mathscr G} (n_{\mathbf f}(\mathscr D^n)) 
&=& 
\pi^{-n} \int_{\mathscr D^n} \ \omega_{\mathscr G}^{\wedge n}
= 
n! \left(\pi^{-1} \int_{\mathscr D} \ \omega_{\mathscr E \mathscr P_d}\right)^n
\\
&=&
n!( d/2+  
0.202,918,921,282 \cdots)^n
.
\end{eqnarray*}

\subsection{An unmixed example}
We consider now the case where the first equation belongs to
$\mathscr G = (\mathscr E \mathscr P_{d_1})^{\otimes n}$ as above, 
but the other equations are polynomials of
degree $d_2, \cdots, d_n$ in each variable (they belong to
$\mathscr P_{d_j}^{\otimes n}$).

Then, let $\mathscr H = \mathscr G^{\lambda_1} \mathscr P_{d_2}^{\lambda_2}
\cdots \mathscr P_{d_n}^{\lambda_n}$.
Note that
\[
\mathscr H = \mathscr E^{\lambda_1} \mathscr P_1^{\lambda_1 d_1 + \cdots + \lambda_n d_n}
.\]
From the previous example,
\begin{eqnarray*}
\frac{1}{n!}\mathbb E_{f_1, \dots, f_n \in \mathscr H} (n_{\mathbf f}(\mathscr D^n)) = 
\hspace{-5em}&&
\hspace{+3.5em} 
\frac{1}{n!\pi^{n}} \int_{\mathscr D^n} \ \omega_{\mathscr H}^{\wedge n}
\\
&=&
\left(\pi^{-1} \int_{\mathscr D} \ \omega_{\mathscr E^{\lambda_1} 
\mathscr P_1^{ \lambda_1 d_1 + \cdots + \lambda_n d_n }}\right)^n
 \\
&=& 
\left( \frac{ \lambda_1 d_1 + \cdots + \lambda_n d_n}{2}
+ \lambda_1 
0.202,918,921,282 \cdots
\right)^n
.\end{eqnarray*}

The coefficient of $\lambda_1 \lambda_2 \cdots \lambda_n$ is
\[
n! 
\frac{d_1 d_2 \cdots d_n}{2^n}
+
(n-1)! 
\frac{d_2 \cdots d_n}{2^{n-1}}
0.202,918,921,282 \cdots
.\]
By Theorem~\ref{th:main}:
\[
\begin{split}
\mathbb E_{f_1 \in \mathscr G, f_2 \in \mathscr P_2, \cdots, f_n \in \mathscr P_n} n_{\mathscr D}(\mathbf f)
= \hspace{-5em} & \\
&=
n! 
\frac{d_1 d_2 \cdots d_n}{2^n}
+
(n-1)! 
\frac{d_2 \cdots d_n}{2^{n-1}}
0.202,918,921,282 \cdots
.
\end{split}
\]

\section{Acknowledgements}

I would like to thank three anonymous referees that 
provided valuable criticism and pointed out important
references.

\end{document}